\documentclass[12pt,a4paper]{amsart}
\usepackage{amssymb,amsmath}
\usepackage{tikz}%%% for tikz
\usetikzlibrary{decorations.markings}%%% for tikz
\tikzstyle{vertex}=[circle, draw, inner sep=0pt, minimum size=6pt]%%% for tikz
\usepackage{booktabs}
\usepackage{hyperref}
\usepackage{mathrsfs}
\usepackage{blkarray}
\usepackage{setspace}
\usepackage{array}
\usepackage{caption}
\usepackage{graphicx}
\usepackage{float}
\usepackage{placeins}

\newtheorem{thm}{Theorem}[section]
\newtheorem{lemma}[thm]{Lemma}

\theoremstyle{definition}

\newtheorem{definitions}[thm]{Definitions}

\newtheorem{op}[thm]{Open Problem}

\begin{document}

\title [On minimal edge version of doubly resolving sets of a graph ]{On minimal edge version of doubly resolving sets of a graph }
\author{Muhammad Ahmad, Zohaib Zahid, Sohail Zafar.}

\address{University of Management and Technology (UMT), Lahore Pakistan}
\email {m.ahmad150092@gmail.com}
\email {sohailahmad04@gmail.com}
\email {zohaib\_zahid@hotmail.com}
\subjclass[---]{Primary 05C12.}
\keywords{ edge version of metric dimension, edge version of doubly resolving set,  prism graph, $n$-sunlet graph }
\doublespacing
\maketitle
\begin{abstract}
In this paper, we introduce the edge version of doubly resolving set of a graph which is based on the edge distances of the graph. As a main result, we computed the minimum cardinality $\psi_E$ of edge version of doubly resolving sets of family of $n$-sunlet graph $S_n$ and prism graph $Y_n$.

\end{abstract}
\section {\textbf{Introduction and Preliminaries}}
Let us take a graph $G=(V(G), E(G)),$ which is simple, connected and undirected, where its vertex set is $V(G)$ and edge set is $E(G)$. The order of a graph $G$ is $|V(G)|$ and the size of a graph $G$ is $|E(G)|$. The distance $d(a, b)$ between the vertices $a$, $b\in V(G)$ is the length of a shortest path between them. If $d(c, a)\neq d(c, b)$, then the vertex $c\in V(G)$ is said to resolve two vertices $a$ and $b$ of $V(G)$. Suppose that $N=\{n_{1}, n_{2},\ldots,n_{k}\}\subseteq V(G)$ is an ordered set and $m$ is a vertex of $V(G)$, then the representation $r(m, N)$ of $m$ with respect to $N$ is the k-tuple $\big(d(m, n_{1}), d(m, n_{2}),\ldots,d(m, n_{k})\big)$. If different vertices of $G$ have different representations with respect to $N$, then the set $N$ is said to be a resolving set of $G$. The metric basis of $G$ is basically a resolving set having minimum cardinality. The cardinality of metric basis is represented by $\dim (G)$, and is called metric dimension of $G.$\\
In \cite{S}, Slater introduced the idea of resolving sets and also in \cite{HM}, Harary and Melter introduced this concept individually. Different applications of this idea has been introduced in the fields like network discovery and verification \cite{BEH}, robot navigation \cite{LG} and chemistry.\\ The introduction of doubly resolving sets is given by Caceres et al. (see \cite{CHM}) by presenting its connection with metric dimension of the cartesian product $G\Box G$ of the graph $G$.\\
The doubly resolving sets create a valuable means for finding upper bounds on the metric dimension of graphs. The vertices $a$ and $b$ of the graph $G$ with order $|V(G)|\geq 2$ are supposed to doubly resolve vertices $u_{1}$ and $v_{1}$ of the graph $G$ if $d(u_{1}, a)-d(u_{1}, b)\neq d(v_{1}, a)-d(v_{1}, b).$ A subset $D$ of vertices doubly resolves $G$ if every two vertices in $G$ are doubly resolved by some two vertices of $D$. Precisely, in $G$ there do not exist any two different vertices having the same difference between their corresponding metric coordinates with respect to $D.$ A doubly resolving set with minimum cardinality is called the minimal doubly resolving set.  The minimum cardinality of a doubly resolving set for $G$ is represented by $\psi(G).$
In case of some convex polytopes, hamming and prism graphs, the minimal doubly resolving sets has been obtained in \cite{KCS1}, \cite{KCS} and \cite{CKS} respectively.\\
Clearly, if $a$ and $b$ doubly resolve $u_{1}$ and $v_{1},$ then $d(u_{1}, a)-d(v_{1}, a)\neq 0$ or $d(u_{1}, b)-d(v_{1}, b)\neq 0,$ and thus $a$ or $b$ resolve $u_{1}$ and $v_{1},$ this shows that a doubly resolving set is also a resolving set, which implies $\dim(G)\leq \psi(G)$ for all graphs $G$. Finding $\psi(G)$ and $\dim(G)$ are NP-hard problems proved in \cite{KR,KC}.\\
Since, the line graph $L(G)$ of a graph $G$ is defined as, the graph whose vertices are the edges of $G$, with two adjacent vertices if the corresponding edges have one vertex common in $G$. In mathematics, the metric properties of line graph have been studied to a great extent (see \cite{B,G,GP,RG,RG1}) and in chemistry literature, its significant applications have been proved (see \cite{GE,GT,GF}). In \cite{NZZ}, the edge version of metric dimension have been introduced, which is defined as:

\begin{definitions}
\begin{enumerate}
  \item The edge distance $d_E(f, g)$ between two edges $f, g \in E(G)$ is the length of a shortest path between vertices $f$ and $g$ in the line graph $L(G).$
  \item   If $d_E(e, f)\neq d_E(e, g)$, then the edge $e\in E(G)$ is said to edge resolve two edges $f$ and $g$ of $E(G)$.
  \item    Suppose that $N_E=\{f_{1}, f_{2},\ldots, f_{k}\}\subseteq E(G)$ is an ordered set and $e$ is an edge of $E(G)$, then the edge version of representation $r_E(e, N_E)$ of $e$ with respect to $N_E$ is the k-tuple $\big(d_E(e, f_{1}), d_E(e, f_{2}),\ldots,d_E(e, f_{k})\big)$.
  \item If different edges of $G$ have different edge version of representations with respect to $N_E$, then the set $N_E$ is said to be a an edge version of resolving set of $G$.
  \item  The edge version of metric basis of $G$ is basically an edge version of resolving set having minimum cardinality. The cardinality of edge version of metric basis is represented by $\dim_E (G)$, and is called edge version of metric dimension of $G.$

\end{enumerate}
\end{definitions}

The following theorems in \cite{NZZ} are important for us.
\begin{thm}\label{t2} Let $S_n$ be the family of $n$-sunlet graph then

$$\,\,\ \dim_E(S_{n})= \left\{ \begin{array}{ll}

2,             & {\rm if}\ n\text{ is even};   \\[2mm]

3 ,     & {\rm if}\  n \text{ is odd.}     \\[2mm]

\end{array} \right.$$

\end{thm}

\begin{thm}\label{t1} Let $Y_n$ be the family of prism graph then $\dim_E(Y_n)=3$ for $n\ge3.$
\end{thm}

In this article, we proposed minimal edge version of doubly resolving sets of a graph $G$, based on edge distances of graph $G$ as follows:
\begin{definitions}
\begin{enumerate}
\item  The edges $f$ and $g$ of the graph $G$ with size $|E(G)|\geq 2$ are supposed to edge doubly resolve edges $f_{1}$ and $f_{2}$ of the graph $G$ if $d_E(f_{1}, f)-d_E(f_{1}, g)\neq d_E(f_{2}, f)-d_E(f_{2}, g)$.
\item    Let $D_E = \{f_1, f_2,\ldots, f_k\}$ be an ordered set of the edges of $G$ then if any two edges $e\neq f\in E(G)$ are edge doubly resolved by some two edges of set $D_E$ then the set $D_E\subseteq E(G)$ is said to be an edge version of doubly resolving set of $G$. The minimum cardinality of an edge version of doubly resolving set of $G$ is represented by $\psi_E (G).$
\end{enumerate}
\end{definitions}

Note that every edge version of doubly resolving set is an edge version of resolving set, which implies $\dim_E(G)\leq \psi_E(G)$ for all graphs $G$.
\section{The edge version of doubly resolving sets for family of $n$-sunlet graph $S_{n}.$}
The family of $n$-sunlet graph $S_{n}$ is obtained by joining $n$ pendant edges to a cycle graph $C_n$ (see Figure \ref{sfig1}).\\

\begin{center}
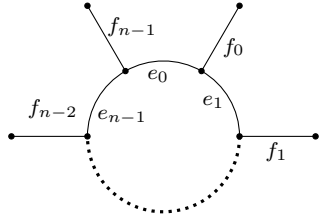

\begin{tikzpicture}[rotate=60]

\draw  (60:1cm) arc[radius=1cm, start angle=60, end angle= 0];

 \filldraw[black] (60:2cm) circle(1pt);
\filldraw[black] (60:1cm) circle(1pt);
 \filldraw[black] (0:2cm) circle(1pt);
\filldraw[black] (0:1cm) circle(1pt);
\draw (0:1cm) -- (0:2cm);
\draw (60:1cm) -- (60:2cm);
  \draw (46:1.5cm)node {\tiny$f_{n-1} $};
  \draw (35:0.8cm)node {\tiny$ e_0$};
  \draw (-8:1.5cm)node {\tiny$f_{0} $};
  \draw (-25:0.8cm)node {\tiny$ e_1$};
   \draw (-68:1.5cm)node {\tiny$f_{1} $};
   \draw (106:1.5cm)node {\tiny$f_{n-2} $};
     \draw (95:0.6cm)node {\tiny$ e_{n-1}$};

\draw  (0:1cm) arc[radius=1cm, start angle=0, end angle= -60];

 \filldraw[black] (-60:2cm) circle(1pt);
\filldraw[black] (-60:1cm) circle(1pt);
 \filldraw[black] (120:2cm) circle(1pt);
\filldraw[black] (120:1cm) circle(1pt);

\draw [very thick,dotted, black] (-60:1cm) arc[radius=1cm, start angle=-60, end angle= -240];

 \draw [black]  (-240:1cm) arc[radius=1cm, start angle=-240, end angle= -300];
\draw [black] (-60:1cm) -- (-60:2cm);
\draw (-240:1cm) -- (-240:2cm);

\end{tikzpicture}\captionof{figure}{$n$-sunlet graph $S_n$ }\label{sfig1}
\end{center}

For our purpose, we label the inner edges of $S_{n}$ by \{$e_i$ : $\forall$ $0\leq i\leq n-1$\} and the pendent edges by \{$f_i$ : $\forall$ $0\leq i\leq n-1$\} as shown in Figure \ref{sfig1}.\\

\begin{center}
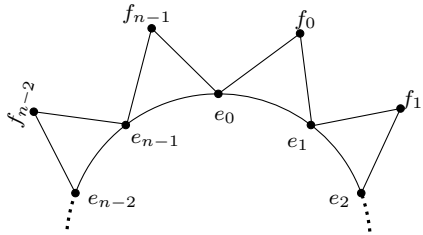

 \begin{tikzpicture}[scale=2]
 \draw  (20:1cm) arc[radius=1cm, start angle=20, end angle= 160];
 \draw [very thick,dotted]  (160:1cm) arc[radius=1cm, start angle=160, end angle=175];
\draw [very thick,dotted]  (5:1cm) arc[radius=1cm, start angle=5, end angle=20];
\foreach \x in {20,52.5,90,127.5,160} {
                     % dots at each point
                \filldraw[black] (\x:1cm) circle(0.7pt);}
                \draw (20:1cm) -- (37:1.5cm);
\draw (52:1cm) -- (37:1.5cm);
\draw (52:1cm) -- (69:1.5cm);
\draw (90:1cm) -- (69:1.5cm);
\draw (90:1cm) -- (107:1.5cm);
\draw (127:1cm) -- (107:1.5cm);
\draw (127:1cm) -- (144:1.5cm);
\draw (160:1cm) -- (144:1.5cm);
\draw (52:.85cm)node {\tiny$e_{1}$};
\draw (20:.85cm)node {\tiny$e_{2}$};
\draw (121:.8cm)node {\tiny$e_{n-1}$};
\draw (158:.75cm)node {\tiny$e_{n-2}$};
\draw (87:.85cm)node {\tiny$e_{0}$};
\draw (37:1.6cm)node {\tiny$f_{1}$};
\draw (107:1.6cm)node {\tiny$f_{n-1}$};
\draw (144:1.6cm)node {\rotatebox{75}{\tiny$f_{n-2}$}};
\draw (69:1.6cm)node {\tiny$f_{0}$};
\foreach \x in {37,69,107,144} {
                     % dots at each point
                \filldraw[black] (\x:1.5cm) circle(0.7pt);}

 \end{tikzpicture}
  \captionof{figure}{$L(S_{n})$ of $n$-sunlet graph $S_{n}$}\label{sfig2}
 \end{center}

 As motivated by the Theorem \ref{t2}, we obtain $$\,\,\ \psi_E (S_{n})\geq  \left\{ \begin{array}{ll}

2,             & {\rm if}\ n\text{ is even};   \\[2mm]

3 ,     & {\rm if}\  n \text{ is odd.}     \\[2mm]

\end{array} \right.$$  Furthermore, we will show that $\psi_E (S_{n})=3$ for $n\geq4.$  \\

In order to calculate the edge distances for family of $n$-sunlet graphs $S_n$, consider the line graph $L(S_n)$ as shown in Figure \ref{sfig2}.\\
Define $S_{i}(e_{0})=\{e\in E(S_{n}):d(e_{0}, e)= i\}$. For $\psi_E (S_{n})$ with $n\geq 4,$ we can locate the sets $S_{i}(e_{0})$ that are represented in the Table \ref{stb1}. It is clearly observed from Figure \ref{sfig2} that $S_{i}(e_{0})= \emptyset$ when $i\geq k+1$ for $n=2k,$ and $S_{i}(e_{0})= \emptyset$ when $i\geq k+2$ for $n=2k+1$. From the above mentioned sets $S_{i}(e_{0}),$  it is clear that they can be utilized to define the edge distances between two arbitrary edges of $E(S_{n})$ in the subsequent way.

\begin{table}[h]
\centering
\caption{$S_{i}(e_{0})$ for $S_{n}$}\label{stb1}
\vspace{1ex}
\begin{tabular}{|l|l|l|}\hline
$n$ & $i$ & $S_{i}(e_{0})$\\ \hline
  & $1 \leq i\leq k$ & \{$f_{i-1}$, $e_{i}$, $f_{n-i}$, $e_{n-i}$\}\\ \hline
$2k (k\geq 2)$ & $k$ & $\{f_{k-1}, f_{k}, e_{k}\}$\\ \hline
$2k+1 (k\geq 2)$ & $k$ & $\{f_{k-1}, e_{k}, f_{k+1}, e_{k+1}\}$\\ \hline
                & $k+1$ & $\{f_{k}\}$\\ \hline
\end{tabular}
\end{table}
\vspace{1ex}
The symmetry in Figure \ref{sfig2} shows that $ d_E(e_{i}, e_{j})= d_E(e_{0}, e_{|j-i|})$ for $0 \leq |j-i|\leq n-1.$ If $n= 2k$, where $k\geq 2,$ we have

$\,\,\ d_E(f_{i}, f_{j})= \left\{ \begin{array}{ll}

d_E(e_{0}, f_{|j-i|})-1,  & {\rm if}\  |j-i|= 0;     \\[2mm]
d_E(e_{0}, f_{|j-i|}),    & {\rm if}\  1\leq |j-i|< k;     \\[2mm]
d_E(e_{0}, f_{|j-i|})+1,    & {\rm if}\  k\leq |j-i|\leq n-1,     \\[2mm]

\end{array} \right.$\\
\vspace{2ex}

$\,\,\ d_E(e_{i}, f_{j})= \left\{ \begin{array}{ll}

d_E(e_{0}, f_{|j-i|}),    & {\rm if}\  0\leq |j-i|\leq n-1 \text{ for }  i\leq j;     \\[2mm]
d_E(e_{0}, f_{|j-i|})-1,    & {\rm if}\  1\leq |j-i|< k  \text{ for }  i> j  ;     \\[2mm]
d_E(e_{0}, f_{|j-i|}),  & {\rm if}\  |j-i|= k  \text{ for }  i> j ;     \\[2mm]
d_E(e_{0}, f_{|j-i|})+1,    & {\rm if}\  k < |j-i|\leq n-1  \text{ for }  i> j.    \\[2mm]

\end{array} \right.$\\
\vspace{2ex}

If $n=2k+1$ where $k\geq 2,$ we have\\

$\,\,\ d_E(f_{i}, f_{j})= \left\{ \begin{array}{ll}

d_E(e_{0}, f_{|j-i|})-1,    & {\rm if}\  |j-i|=0;     \\[2mm]
d_E(e_{0}, f_{|j-i|}),   & {\rm if}\  1\leq |j-i|\leq k;     \\[2mm]
d_E(e_{0}, f_{|j-i|})+1,    & {\rm if}\  k< |j-i|\leq n-1,     \\[2mm]

\end{array} \right.$\\
\vspace{2ex}

$\,\,\ d_E(e_{i}, f_{j})= \left\{ \begin{array}{ll}

d_E(e_{0}, f_{|j-i|}),    & {\rm if}\  0\leq |j-i|\leq n-1  \text{ for }   i\leq j;     \\[2mm]
d_E(e_{0}, f_{|j-i|})-1,    & {\rm if}\  1\leq |j-i|\leq k  \text{ for }   i> j  ;     \\[2mm]
d_E(e_{0}, f_{|j-i|})+1,    & {\rm if}\  k < |j-i|\leq n-1  \text{ for }   i> j.     \\[2mm]

\end{array} \right.$\\

As a result, if we know the edge distance $d_E(e_{0},e)$ for any $ e\in E(S_{n})$, then one can recreate the edge distances between any two edges from $E(S_n).$

\begin{lemma}\label{l1}
$ \psi_E(S_{n})> 2$, for $n=2k$, $k\geq 2.$
\end{lemma}
\begin{proof}
As we know that for $n =2k$, $\psi_E(S_{n})\geq 2.$ So it is necessary to prove that each of the subset $D_E$ of edge set $E(S_{n})$ such that $|D_E|=2$ is not an edge version of doubly resolving set for $S_{n}.$ In Table \ref{stb2}, seven possible types of set $D_E$ are presented and for each of them the resultant non-edge doubly resolved pair of edges from edge set $E(S_{n})$ is found. To verify, let us take an example, the edges $ e_{k}, e_{k+1} $ are not edge doubly resolved by any two edges of the set $\{e_{0}, e_{i}; k< i \leq n-1\}$.
Obviously, for $ k< i\leq n-1 $, we have\\
$d_E(e_{0}, e_{k})= d_E(e_{0}, e_{|k-0|})= k$, $d_E(e_{0}, e_{k+1})= d_E(e_{0}, e_{|k+1-0|})= k-1$, $d_E(e_{i}, e_{k})= d_E(e_{0}, e_{|k-i|})= i-k$ and
$d_E(e_{i}, e_{k+1})=d_E(e_{0}, e_{|k+1-i|})=i-k-1$. So, $d_E(e_{0}, e_{k})- d_E(e_{0}, e_{k+1})= d_E(e_{i}, e_{k})- d_E(e_{i}, e_{k+1})= 1,$ that is,
$\{e_{0}, e_{i}; k< i \leq n-1\}$ is not an edge version of doubly resolving set of $S_{n}$. Using this procedure we can verify all other non-edge doubly resolved pairs of edges for all other possible types of $D_E$ from Table \ref{stb2}.

\begin{table}[h]
\centering
\caption{Non-edge doubly resolved pairs of $S_{n}$ for $n=2k,$ $k\geq 2$}\label{stb2}
\vspace{1ex}
\begin{tabular}{|l|l|}\hline
$D_E$ & Non-edge doubly resolved pairs \\ \hline
$\{e_{0}, e_{i}\}$, $0 < i < k$  & $\{e_{0}, e_{n-1}\}$\\ \hline
$\{e_{0}, e_{i}\}$, $k < i \leq n-1$  & $\{e_{k}, e_{k+1}\}$\\ \hline
$\{e_{0}, f_{i}\}$, $0 \leq i < k$  & $\{e_{0}, f_{n-1}\}$\\ \hline
$\{e_{0}, f_{i}\}$, $k \leq i \leq n-1$  & $\{e_{0}, f_{0}\}$\\ \hline
$\{f_{0}, f_{i}\}$, $1 \leq i < k$  & $\{e_{k}, f_{k}\}$\\ \hline
$\{f_{0}, f_{k}\}$    & $\{e_{0}, e_{1}\}$\\ \hline
$\{f_{0}, f_{i}\}$, $k < i \leq n-1$  & $\{e_{1}, f_{1}\}$\\ \hline
\end{tabular}
\end{table}
\end{proof}
\begin{lemma}\label{sl2}
$\psi_E(S_{n})=3,$ for $n=2k$, $k\geq 2.$
\end{lemma}
\begin{proof}
The Table \ref{stb3} demonstrate that edge version of representations of $S_{n}$ in relation to the set $D_E^{\ast} = \{e_{0}, e_{1}, e_{k}\}$ in a different manner.\\

\begin{table}[h!]
\centering
\caption{Vectors of edge metric coordinates for $S_{n},$ $n=2k,$ $k\geq 2$}\label{stb3}
\vspace{1ex}
\begin{tabular}{|l|l|l|}\hline
$i$ & $ S_{i}(e_{0})$ & $D_E^{\ast}= \{{e_{0}, e_{1}, e_{k}}\}$ \\ \hline
$0$ & $e_{0}$  & $(0, 1, k)$\\ \hline
$1\leq i< k$  & $f_{i-1}$  & $(i, i-1, k+1-i)$\\
    & $e_{i}$  & $(i, i-1, k-i)$\\
    & $f_{n-i}$  & $(i, i+1, k+1-i)$\\
    & $e_{n-i}$ & $(i, i+1, k-i)$\\ \hline
$i=k$ & $f_{k-1}$  & $(k, k-1, 1)$\\
      & $f_{k}$  & $(k, k, 1)$\\
      & $e_{k}$  & $(k, k-1, 0)$\\ \hline
\end{tabular}
\end{table}
\vspace{1ex}
Now from Table \ref{stb3}, as $e_{0} \in D_E^{\ast}$, so the first edge version of metric coordinate of the vector of $e_{0}\in S_{i}(e_{0})$ is equal to $0$. For each $i\in \{1,2,3,\ldots,k\}$, one can easily check that there are no two edges $h_1, h_2 \in S_{i}(e_{0})$ such that $ r_E(h_1, D_E^{\ast})- r_E(h_2, D_E^{\ast})=0$. Also, for each $i, j\in \{1,2,3,\ldots,k\}, i\neq j$, there are no two edges $h_1 \in S_{i}(e_{0})$ and $h_2 \in S_{j}(e_{0})$ such that $ r_E(h_1,D_E^{\ast})- r_E(h_2,D_E^{\ast})=i- j$. In this manner, the set $D_E^{\ast}= \{{e_{0}, e_{1}, e_{k}}\}$ is the minimal edge version of doubly resolving set for $S_{n}$ with $n=2k$, $k\geq 2$ and hence Lemma \ref{sl2} holds.
\end{proof}
\vspace{1ex}
\begin{lemma}\label{sl3}
$\psi_E(S_{n})=3$, for $n=2k+1$, $k\geq 2.$
\end{lemma}
\begin{proof}
The Table \ref{stb4} demonstrate that the edge version of representations of $S_{n}$ in relation to the set $D_E^{\ast} = \{e_{0}, e_{1}, e_{k+1}\}$ in a different way.\\

\begin{table}
\centering
\caption{Vectors of edge metric coordinates for $S_{n},$ $n=2k+1,$ $k\geq 2$}\label{stb4}
\vspace{1ex}
\begin{tabular}{|l|l|l|}\hline
$i$ & $ S_{i}(e_{0})$ & $D_E^{\ast}= \{{e_{0}, e_{1}, e_{k+1}}\}$ \\ \hline
$0$ & $e_{0}$  & $(o, 1, k)$\\ \hline
$1\leq i< k$ & $f_{i-1}$ & $(i, i-1, k+2-i)$\\
    & $e_{i}$ & $(i, i-1, k+1-i)$\\
    & $f_{n-i}$ & $(i, i+1, k+1-i)$\\
    & $e_{n-1}$  & $(i, i+1, k-i)$\\ \hline
$i=k$ & $f_{k-1}$  & $(k, k-1, 2)$\\
        & $e_{k}$  & $(k, k-1, 1)$\\
        & $f_{k+1}$  & $(k, k+1, 1)$\\
        & $e_{k+1}$  & $(k, k, 0)$\\ \hline
$i=k+1$ & $f_{k}$  & $(k+1, k, 1)$\\ \hline
\end{tabular}
\end{table}
\vspace{1ex}
Now from Table \ref{stb4}, as $e_{0} \in D_E^{\ast}$, so the first edge version of metric coordinate of the vector of $e_{0}\in S_{i}(e_{0})$ is equal to $0$. Similarly for each $i\in \{1,2,3,\ldots,k+1\}$, one can easily find that there are no two edges $h_1, h_2 \in S_{i}(e_{0})$ such that $ r_E(h_1,D_E^{\ast})- r_E(h_2,D_E^{\ast})=0$. Likewise, for every $i, j\in \{1,2,3,\ldots,k+1\}, i\neq j$, there are no two edges $h_1 \in S_{i}(e_{0})$ and $h_2 \in S_{j}(e_{0})$ such that  $r_E(h_1,D_E^{\ast})- r_E(h_2,D_E^{\ast})=i- j$. Like so, the set $D_E^{\ast}= \{{e_{0}, e_{1}, e_{k+1}}\}$ is the minimal edge version of doubly resolving set for $S_{n}$ with $n=2k+1$, $k\geq 2$ and consequently Lemma \ref{sl3} holds.\\
\end{proof}
It is displayed from the whole technique that $\psi_E(S_{n})=3$, for $n\geq 4$. We state the resulting main theorem by using Lemma \ref{sl2} and Lemma \ref{sl3} as mentioned below;\\

\begin{thm}
Let $S_{n}$ be the $n$-sunlet graph for $n\geq 4.$ Then $\psi_E(S_{n})=3.$
\end{thm}
\section{The edge version of doubly resolving sets for family of prism graph $Y_{n}.$}
A family of prism graph $Y_n$ is cartesian product graph $C_n\times P_2$, where $C_n$ is cycle graph of order $n$ and $P_2$ is a path of order $2$ (see Figure \ref{fg1}).
 \begin{center}
 \begin{tikzpicture}[scale=1.8]
 \draw  (0:1.6cm) arc[radius=1.6cm, start angle=0, end angle= 180];
 \draw  (0:2.5cm) arc[radius=2.5cm, start angle=0, end angle= 180];
 \draw [very thick,dotted]  (180:1.6cm) arc[radius=1.6cm, start angle=180, end angle=185];
 \draw [very thick,dotted]  (355:1.6cm) arc[radius=1.6cm, start angle=355, end angle=360];
 \draw [very thick,dotted]  (180:2.5cm) arc[radius=2.5cm, start angle=180, end angle=185];
 \draw [very thick,dotted]  (355:2.5cm) arc[radius=2.5cm, start angle=355, end angle=360];
                \foreach \x in {3,37,69,107,144,177} {
                     % dots at each point
                \filldraw[black] (\x:1.6cm) circle(0.7pt);}
                  \foreach \x in {3,37,69,107,144,177} {
                % lines from center to point
                 \draw[black] (\x:1.6cm) -- (\x:2.5cm);}
                 \foreach \x in {3,37,69,107,144,177} {
                     % dots at each point
                \filldraw[black] (\x:2.5cm) circle(0.7pt);}
\draw (52:1.5cm)node {\tiny$e_{1}$};
\draw (20:1.5cm)node {\tiny$e_{2}$};
\draw (121:1.4cm)node {\tiny$e_{n-1}$};
\draw (158:1.4cm)node {\tiny$e_{n-2}$};
\draw (87:1.5cm)node {\tiny$e_{0}$};
\draw (7:1.9cm)node {\tiny$f_{3}$};
\draw (34:1.9cm)node {\tiny$f_{2}$};
\draw (66:1.9cm)node {\tiny$f_{1}$};
\draw (102:1.9cm)node {\tiny$f_{0}$};
\draw (138:1.9cm)node {\rotatebox{71}{\tiny$f_{n-1}$}};
\draw (170:1.9cm)node {\rotatebox{75}{\tiny$f_{n-2}$}};
\draw (52:2.65cm)node {\tiny$g_{1}$};
\draw (20:2.65cm)node {\tiny$g_{2}$};
\draw (121:2.7cm)node {\tiny$g_{n-1}$};
\draw (158:2.7cm)node {\tiny$g_{n-2}$};
\draw (87:2.65cm)node {\tiny$g_{0}$};

\end{tikzpicture}
 
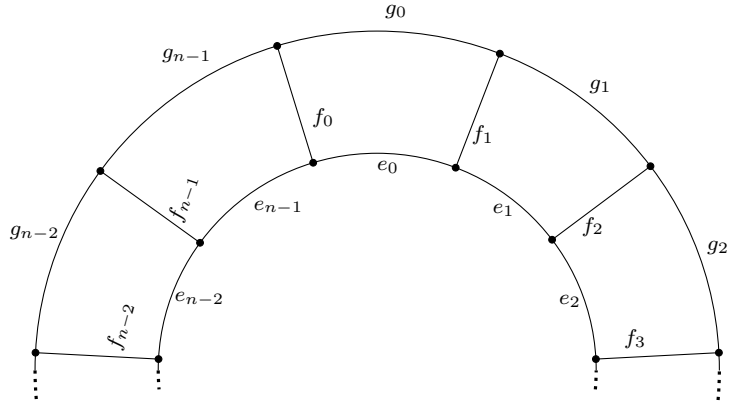
\captionof{figure}{Prism graph $Y_{n}$}\label{fg1}
 \end{center}
The family of prism graph $Y_{n}$ consists of $4$-sided faces and $n$-sided faces. For our purpose, we label the inner cycle edges of $Y_{n}$ by \{$e_i$ :  $0\leq i\leq n-1$\}, middle edges by \{$f_i$ : $0\leq i\leq n-1$\} and the outer cycle edges by \{$g_i$ :  $1\leq i\leq n-1$\} as shown in Figure \ref{fg1}.

 \begin{center}
 \begin{tikzpicture}[scale=2.2]
\draw  (20:1cm) arc[radius=1cm, start angle=20, end angle= 160];
\draw  (10:1.9cm) arc[radius=1.9cm, start angle=10, end angle= 170];
\draw [very thick,dotted]  (160:1cm) arc[radius=1cm, start angle=160, end angle=175];
\draw [very thick,dotted]  (5:1cm) arc[radius=1cm, start angle=5, end angle=20];
 \draw [very thick,dotted]  (170:1.9cm) arc[radius=1.9cm, start angle=170, end angle=180];
 \draw [very thick,dotted]  (0:1.9cm) arc[radius=1.9cm, start angle=0, end angle=10];
\foreach \x in {20,52.5,90,127.5,160} {
                     % dots at each point
                \filldraw[black] (\x:1cm) circle(0.7pt);}
                \foreach \x in {10,30,58,90,123,150,170} {
                     % dots at each point
                \filldraw[black] (\x:1.9cm) circle(0.7pt);}
                \foreach \x in {15.5,41,73,107,139,164.5} {
                     % dots at each point
                \filldraw[black] (\x:1.28cm) circle(0.7pt);}
\draw (20:1cm) -- (58:1.9cm);
\draw (52.5:1cm) -- (30:1.9cm);
\draw (90:1cm) -- (58:1.9cm);
\draw (52.5:1cm) -- (90:1.9cm);
\draw (90:1cm) -- (123:1.9cm);
\draw (127.5:1cm) -- (90:1.9cm);
\draw (160:1cm) -- (123:1.9cm);
\draw (127.5:1cm) -- (150:1.9cm);
\draw (20:1cm) -- (10:1.9cm);
\draw (160:1cm) -- (170:1.9cm);
\draw (15.5:1.28cm) -- (30:1.9cm);
\draw (164.5:1.28cm) -- (150:1.9cm);
\draw [very thick,dotted]  (164.5:1.28cm) -- (178:1cm);
\draw [very thick,dotted]  (15.5:1.28cm) -- (5:1cm);
\draw (52:.85cm)node {\tiny$e_{1}$};
\draw (20:.85cm)node {\tiny$e_{2}$};
\draw (121:.8cm)node {\tiny$e_{n-1}$};
\draw (158:.75cm)node {\tiny$e_{n-2}$};
\draw (87:.85cm)node {\tiny$e_{0}$};
\draw (11:1.3cm)node {\tiny$f_{3}$};
\draw (36:1.29cm)node {\tiny$f_{2}$};
\draw (67:1.28cm)node {\tiny$f_{1}$};
\draw (100:1.28cm)node {\tiny$f_{0}$};
\draw (129:1.31cm)node {\rotatebox{71}{\tiny$f_{n-1}$}};
\draw (154:1.28cm)node {\rotatebox{75}{\tiny$f_{n-2}$}};
\draw (10:2.05cm)node {\tiny$g_{3}$};
\draw (30:2.0cm)node {\tiny$g_{2}$};
\draw (58:2.0cm)node {\tiny$g_{1}$};
\draw (90:2.0cm)node {\tiny$g_{0}$};
\draw (123:2.05cm)node {\tiny$g_{n-1}$};
\draw (150:2.1cm)node {\tiny$g_{n-2}$};
\draw (170:2.1cm)node {\tiny$g_{n-3}$};
\end{tikzpicture}

 
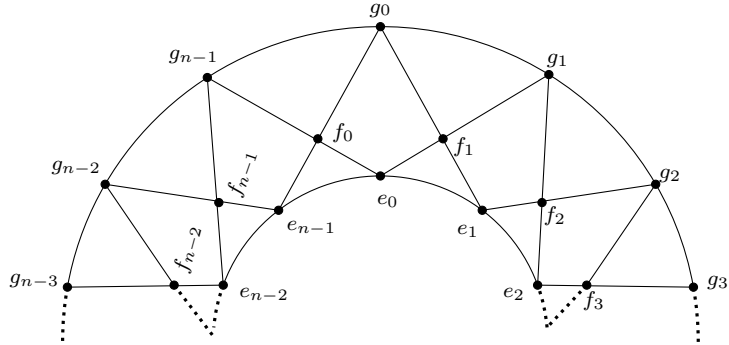
\captionof{figure}{$L(Y_{n})$ of prism graph $Y_{n}$}\label{fig2}
 \end{center}

 As motivated by the Theorem \ref{t1}, we obtain $\psi_E (Y_{n})\geq 3$. Furthermore, we will show that $\psi_E (Y_{n})=3$ for $n\geq6.$  \\

In order to calculate the edge distances for family of prism graphs $Y_n$, consider the line graph $L(Y_n)$  as shown in Figure \ref{fig2}.\\
Define $S_{i}(f_{0})=\{f\in E(Y_{n}):d_E(f_{0}, f)= i\}$. For $\psi_E (Y_{n})$ with $n\geq 6,$ we can locate the sets $S_{i}(f_{0})$ that are represented in the Table \ref{tb1}. It is clearly observed from Figure \ref{fig2} that $S_{i}(f_{0})= \emptyset$ for $i\geq k+2.$ From the above mentioned sets $S_{i}(f_{0}),$  it is clear that they can be utilized to define the edge distance between two arbitrary edges of $E(Y_{n})$ in the subsequent way.
\begin{table}[h]
\centering
\caption{$S_{i}(f_{0})$ for $Y_{n}$}\label{tb1}
\vspace{1ex}
\begin{tabular}{|l|l|l|}\hline

$n$ & $i $&$ S_{i}(f_{0})$\\ \hline
  & 1 & $\{e_{0}, g_{0}, e_{n-1}, g_{n-1}\}$\\ \hline
  & 2 $\leq i\leq k$& $\{f_{i-1}, e_{i-1}, g_{i-1}, f_{n+1-i}, e_{n-i}, g_{n-i}$\}\\ \hline
$2k(k\geq 3)$ &$ k+1$ & $\{f_{k}\}$\\ \hline
$2k+1(k\geq 3)$ &$ k+1 $& $\{f_{k}, e_{k}, g_{k}, f_{k+1}\}$\\ \hline
\end{tabular}
\end{table}\\
\vspace{1ex}
The symmetry in Figure \ref{fig2} shows that $ d_E(f_{i}, f_{j})= d_E(f_{0}, f_{|j-i|})$ for $0 \leq |j-i|\leq n-1.$ If $n= 2k$, where $k\geq 3,$ we have

$\,\,\ d_E(e_{i}, e_{j})= d_E(g_{i}, g_{j})= \left\{ \begin{array}{ll}

d_E(f_{0}, e_{|j-i|})-1, & {\rm if}\  0\leq |j-i|< k;     \\[2mm]
d_E(f_{0}, e_{|j-i|}),  & {\rm if}\  k\leq |j-i|\leq n-1,     \\[2mm]

\end{array} \right.$\\
\vspace{2ex}

$\,\,\ d_E(f_{i}, e_{j})= d_E(f_{i}, g_{j})= \left\{ \begin{array}{ll}

d_E(f_{0}, e_{|j-i|}),    & {\rm if}\  0\leq |j-i|\leq n-1,  \text{ for } i\leq j;     \\[2mm]
d_E(f_{0}, e_{|j-i|})-1,    & {\rm if}\  1\leq |j-i|< k,  \text{ for }  i> j;     \\[2mm]
d_E(f_{0}, e_{|j-i|}),            & {\rm if}\  |j-i|=k,  \text{ for }  i> j;     \\[2mm]
d_E(f_{0}, e_{|j-i|})+1,    & {\rm if}\  k < |j-i|\leq n-1,  \text{ for }  i> j,     \\[2mm]

\end{array} \right.$\\
\vspace{2ex}

$\,\,\, d_E(e_{i}, g_{j})= \left\{ \begin{array}{ll}

d_E(f_{0}, e_{|j-i|})+1,  & {\rm if}\  |j-i|=0;     \\[2mm]
d_E(f_{0}, e_{|j-i|}),    & {\rm if}\  1\leq |j-i|< k;     \\[2mm]
d_E(f_{0}, e_{|j-i|})+1,    & {\rm if}\  k \leq |j-i|\leq n-1.     \\[2mm]

\end{array} \right.$\\
\vspace{2ex}

If $n= 2k+1$ where $k\geq 3,$ we have\\

$\,\,\ d_E(e_{i}, e_{j})= d_E(g_{i}, g_{j})= \left\{ \begin{array}{ll}

d_E(f_{0}, e_{|j-i|})-1,  & {\rm if}\  0\leq |j-i|\leq k;     \\[2mm]
d_E(f_{0}, e_{|j-i|}),    & {\rm if}\  k< |j-i|\leq n-1,     \\[2mm]

\end{array} \right.$\\
\vspace{2ex}

$\,\,\ d_E(f_{i}, e_{j})= d_E(f_{i}, g_{j})= \left\{ \begin{array}{ll}

d_E(f_{0}, e_{|j-i|}),    & {\rm if}\  0\leq |j-i|\leq n-1   \text{ for }   i\leq j;     \\[2mm]
d_E(f_{0}, e_{|j-i|})-1,    & {\rm if}\  1\leq |j-i|\leq k   \text{ for }   i> j  ;     \\[2mm]
d_E(f_{0}, e_{|j-i|})+1,    & {\rm if}\  k < |j-i|\leq n-1   \text{ for }   i> j,     \\[2mm]

\end{array} \right.$\\
\vspace{2ex}

$\,\,\, d_E(e_{i}, g_{j})= \left\{ \begin{array}{ll}

d_E(f_{0}, e_{|j-i|})+1,  & {\rm if}\  |j-i|=0;     \\[2mm]
d_E(f_{0}, e_{|j-i|}),    & {\rm if}\  1\leq |j-i|\leq k;     \\[2mm]
d_E(f_{0}, e_{|j-i|})+1,    & {\rm if}\  k < |j-i|\leq n-1.     \\[2mm]

\end{array} \right.$\\

As a result, if we know the edge distance $d_E(f_{0},f)$ for any $f\in E(Y_n)$ then one can recreate the edge distances between any two edges from $E(Y_n).$

\begin{lemma}\label{l2}
$\psi_E (Y_{n})=3,$ for $n=2k$, $k\geq 3.$
\end{lemma}
\begin{proof}
 The Table \ref{tb3} demonstrate that edge version of representations of $Y_{n}$ in relation to the set $D_E^{\ast} = \{e_{0}, e_{k-1}, f_{k+1}\}$ in a different manner.\\

\begin{table}[h!]
\centering
\caption{Vectors of edge metric coordinates for $Y_{n},$ $n=2k,$ $k\geq 3$}\label{tb3}
\vspace{1ex}
\begin{tabular}{|l|l|l|}\hline
$i$ & $ S_{i}(f_{0})$ & $D_E^{\ast}= \{{e_{0}, e_{k-1}, f_{k+1}}\}$ \\ \hline
$0$ & $f_{0}$  & $(1, k, k)$\\ \hline
$1$ & $e_{0}$  & $(0, k-1, k)$\\
    & $g_{0}$  & $(2, k, k)$\\
    & $e_{n-1}$  & $(1, k, k-1)$\\
    & $g_{n-1}$  & $(2, k+1, k-1)$\\ \hline
$2$ & $f_{1}$  & $(1, k-1, k+1)$\\
    & $e_{1}$  & $(1, k-2, k)$\\
    & $g_{1}$  & $(2, k-1, k)$\\
    & $f_{n-1}$  & $(2, k, k-1)$\\
    & $e_{n-2}$  & $(2, k-1, k-2)$\\
    & $g_{n-2}$  & $(3, k, k-2)$\\ \hline
$3\leq i \leq k$  & $f_{i-1}$  & $(i-1, k+1-i, k+3-i)$\\
    & $e_{i-1}$  & $(i-1, k-i, k+2-i)$\\
    & $g_{i-1}$  & $\,\,\ = \left\{ \begin{array}{ll}
                   $($k$, 2, 2)$, & {\rm if}\     $$i=k$$; \\[2mm]\\
                   $$(i, k+1-i, k+2-i)$$, & {\rm if}\ $$i< k$$. \\[2mm]\\
\end{array} \right.$\\
    & $f_{n+1-i}$ & $\,\,\ = \left\{ \begin{array}{ll}
                  $$(k, 2, 0)$$, & {\rm if}\ $$i=k$$; \\[2mm]\\
                     $$(i, k+2-i, k+1-i)$$, & {\rm if}\ $$i $$<$$ k$$ \\[2mm]\\
\end{array} \right.$\\
  & $e_{n-i}$  &  $\,\,\ = \left\{ \begin{array}{ll}
                  $($k, 1, 1$)$, & {\rm if}\ $$i=k$$; \\[2mm]\\
                     $$(i, k+1-i, k-i)$$, & {\rm if}\ $$i< k$$ \\[2mm]\\
\end{array} \right.$\\
    & $g_{n-i}$  & $\,\,\ = \left\{ \begin{array}{ll}
                   $$(k+1, 2, 1)$$, & {\rm if}\ $$i=k$$; \\[2mm]\\
                     $$(i, k+2-i, k-i)$$, & {\rm if}\ $$i< k$$ \\[2mm]\\
\end{array} \right.$\\ \hline
$i=k+1$ & $f_{k}$  & $(k, 1, 2)$\\\hline
\end{tabular}
\end{table}
\vspace{1ex}
Now from Table \ref{tb3}, as $e_{0} \in D_E^{\ast}$, so the first edge version of metric coordinate of the vector of $f_{0}\in S_{i}(f_{0})$ is equal to $1$. For each $i\in \{1,2,3,\ldots,k+1\}$, one can easily check that there are no two edges $h_1, h_2 \in S_{i}(f_{0})$ such that $ r_E(h_1, D_E^{\ast})- r_E(h_2, D_E^{\ast})=0$. Also, for each $i, j\in \{1,2,3,\ldots,k+1\}, i\neq j$, there are no two edges $h_1 \in S_{i}(f_{0})$ and $h_2 \in S_{j}(f_{0})$ such that $ r_E(h_1,D_E^{\ast})- r_E(h_2,D_E^{\ast})=i- j$. In this manner, the set $D_E^{\ast}= \{{e_{0}, e_{k-1}, f_{k+1}}\}$ is the minimal edge version of doubly resolving set for $Y_{n}$ with $n=2k$, $k\geq 3$ and hence Lemma \ref{l2} holds.
\end{proof}
\vspace{1ex}
\begin{lemma}\label{l3}
$\psi_E (Y_{n})=3$, for $n=2k+1$, $k\geq 3.$
\end{lemma}
\begin{proof}
The Table \ref{tb4} demonstrate that the edge version of representations of $Y_{n}$ in relation to the set $D_E^{\ast} = \{e_{0}, e_{k}, g_{k+2}\}$ in a different way.\\

\begin{table}
\centering
\caption{Vectors of edge metric coordinates for $Y_{n},$ $n=2k+1,$ $k\geq 3$}\label{tb4}
\vspace{1ex}
\begin{tabular}{|l|l|l|}\hline
$i$ & $ S_{i}(f_{0})$ & $D_E^{\ast}= \{{e_{0}, e_{k}, g_{k+2}}\}$ \\ \hline
$0$ & $f_{0}$  & $(1, k+1, k-1)$\\ \hline
$1$ & $e_{0}$  & $(0, k, k)$\\
    & $g_{0}$  & $(2, k+1, k-1)$\\
    & $e_{n-1}$  & $(1, k, k-1)$\\
    & $g_{n-1}$  & $(2, k+1, k-2)$\\ \hline
$2$ & $f_{1}$  & $(1, k, k)$\\
    & $e_{1}$  & $(1, k-1, k+1)$\\
    & $g_{1}$  & $(2, k, k)$\\
    & $f_{n-1}$  & $(2, k, k-2)$\\
    & $e_{n-2}$  & $\,\,\ = \left\{ \begin{array}{ll}
                   (2, 2, 2), & {\rm if}\ k=3; \\[2mm]\\
                   (2, k-i, k-2), & {\rm if}\ k< 3. \\[2mm]\\
\end{array} \right.$\\
    & $g_{n-2}$  & $(3, k, k-3)$\\ \hline
$3\leq i \leq k$  & $f_{i-1}$  & $(i-1, k+2-i, k+4-i)$\\
                  & $e_{i-1}$  & $(i-1, k+1-i, k+4-i)$\\
                  & $g_{i-1}$  & $(i, k+2-i, k+3-i)$\\
    & $f_{n+1-i}$ & $\,\,\ = \left\{ \begin{array}{ll}
                  (k, 2, 1), & {\rm if}\ i=k; \\[2mm]\\
                  (i, k+2-i, k-i), & {\rm if}\ i+1 \leq k \\[2mm]\\
\end{array} \right.$\\
  & $e_{n-i}$  &  $\,\,\ = \left\{ \begin{array}{ll}
                  (k, 1, 2), & {\rm if}\ i=k; \\[2mm]\\
                  (i,2, 2), & {\rm if}\ i+1=k; \\[2mm]\\
                  (i, k+1-i, k-i), & {\rm if}\ i+1 < k \\[2mm]\\
\end{array} \right.$\\
    & $g_{n-i}$  & $\,\,\ = \left\{ \begin{array}{ll}
                   (k+1, 2, 1), & {\rm if}\ i=k; \\[2mm]\\
                     (i+1, k+2-i, k-1-i), & {\rm if}\ i+1 \leq k \\[2mm]\\
\end{array} \right.$\\ \hline
$i=k+1$ & $f_{k}$  & $(k, 1, 3)$\\
        & $e_{k}$  & $(k, 0, 3)$\\
        & $g_{k}$  & $(k+1, 2, 2)$\\
        & $f_{k+1}$  & $(k+1, 1, 2)$\\\hline
\end{tabular}
\end{table}
\vspace{1ex}
    Now from Table \ref{tb4}, as $e_{0} \in D_E^{\ast}$, so the first edge version of metric coordinate of the vector of $f_{0}\in S_{i}(f_{0})$ is equal to $1$.
Similarly for each $i\in \{1,2,3,\ldots,k+1\}$, one can easily find that here are no two edges $h_1, h_2 \in S_{i}(f_{0})$ such that $ r_E(h_1,D_E^{\ast})- r_E(h_2,D_E^{\ast})=0$. Likewise, for every $i, j\in \{1,2,3,\ldots,k+1\}, i\neq j$, there are no two edges $h_1 \in S_{i}(f_{0})$ and $h_2 \in S_{j}(f_{0})$ such that $r_E(h_1,D_E^{\ast})- r_E(h_2,D_E^{\ast})=i- j$. Like so, the set $D_E^{\ast}= \{{e_{0}, e_{k}, g_{k+2}}\}$ is the minimal edge version of doubly resolving set for $Y_{n}$ with $n= 2k+1$, $k\geq 3$ and consequently Lemma \ref{l3} holds.\\
\end{proof}
It is displayed from the whole technique that $\psi_E(Y_{n})=3$, for $n\geq 6$. We state the resulting main theorem by using Lemma \ref{l2} and Lemma \ref{l3} as mentioned below;\\

\begin{thm}
Let $Y_{n}$ be the prism graph for $n\geq 6.$ Then $\psi_E (Y_{n})=3.$
\end{thm}
\section{Conclusion}
In this article, we computed the minimal edge version of doubly resolving sets and its cardinality $\psi_E(G)$ by considering $G$ as a family of $n$-sunlet graph $S_{n}$ and prism graph $Y_{n}$. In case of $n$-sunlet graphs, the graph is interesting to consider in the sense that its edge version of metric dimension $\dim_E(S_{n})$ is dependent on the parity of $n$ for both even and odd cases. The cardinality $\psi_E(S_{n})$ of minimal edge version of doubly resolving set of $n$-sunlet graph $S_{n}$ is independent from the parity of $n$. In the case of prism graph $Y_{n}$, the edge version of metric dimension $\dim_E (Y_{n})$ and the cardinality $\psi_E(Y_{n})$ of its minimal edge version of doubly resolving set are same for every $n\geq 6.$
\begin{op} Compute edge version of doubly resolving sets for some generalized petersen graphs.
\end{op}

\end{document}